\documentclass{article}
\usepackage{amsmath,amssymb,amsthm,stmaryrd,makeidx}
\usepackage[all]{xy}

\DeclareMathOperator{\ann}{Ann}

\DeclareMathOperator{\mor}{Mor}

\DeclareMathOperator{\spec}{Spec}
\DeclareMathOperator{\aspec}{aSpec}

\newcommand{\cat}[1]{\mathbf{#1}}

\newcommand{\grmcat}[1]

\input xypic

\newtheorem{proposition}{Proposition}

\newtheorem{lemma}{Lemma}
\newtheorem{corollary}{Corollary}
\newtheorem{definition}{Definition}

\makeindex

\begin{document}
\author{Arvid Siqveland}
\title{Categorical construction of Schemes}

\maketitle

\begin{abstract} In the book \cite{S23} and the following article \cite{S241} we use an algebraization of the semi-local formal moduli of simple modules to construct associative schemes. Here, we consider a commutative ring for which we can use the localization in maximal ideals as local moduli. This gives a categorical definition of schemes that is equivalent to the definition in Hartshorne's book, \cite{HH77}. The definition includes a construction of the sheaf associated to a presheaf using projective limits, and this makes the basic results in scheme theory more natural.
\end{abstract}

\section{Introduction}

In \cite{S241} we define associative schemes which are ringed spaces covered by $\aspec A$ where $A$ is an associative ring. From deformation theory we construct a localization $A_P$ of the associative ring $A$ in $r$ simple modules $P=\{P_1,\dots,P_r\}.$ We define the set $\aspec A$ as the set of aprime  $A$-modules: An aprime $A$-module is a simple (right) module or a contraction of a simple $B$-module via a homomorphism $f:A\rightarrow B.$ We give $\aspec A$ a topology which equals the Zariski topology in the case where $A$ is commutative. From this, we define a sheaf of associative rings on $\aspec A$ by letting $\mathcal O_X(U)$ be the projective limit over strictly contained open subsets $V\subsetneq U$, of the image of $A$ in the $A_M,$ for $M\subseteq V$ finite. 

The deformation theoretic definition of localization proves that when $A$ is a commutative domain and $\mathfrak m\subset A$ a maximal ideal, $A_{\{A/\mathfrak m\}}\simeq A_{\mathfrak m},$ and so $\aspec A\simeq\spec A.$ If $A$ is finitely generated over a field, it follows that $\spec A$ has finite dimension and has an irreducible open subset and so $\aspec A\simeq\spec A$ for all such algebras $A.$ Here we will define affine schemes of commutative rings by applying projective limits with the local rings $A_{\mathfrak m}$ in general (that is, not the ones coming from deformation theory), proving explicitly that when $A$ is a commutative ring, i.e. not necessarily finitely generated over a field, $\aspec A\simeq\spec A.$ 

This categorical definition of schemes of commutative rings makes the generalization to schemes of associative rings more streamlined. Also, a lot of properties are defined more natural (functorial), making a lot of results follow more natural.

We assume familiarity with Hartshorne's Algebraic Geometry, \cite{HH77}, and we assume knowledge of inductive (direct) and projective (inverse) limits in categories.

\section{Sheaves on Topological spaces}

This section can very well be generalized to any subset of the power set of a set $X.$ We choose to restrict to the categories needed for generalizing to scheme theory of associative rings.

The topology on a topological space $X$ defines a category $\cat{Top}(X)$ where the objects are the open sets, and the morphisms are the inclusions of sets. We will study functors $F:\cat{Top}(X)\rightarrow\cat C$ where $\cat C$ is a category of a kind described below.

A category $\cat C$ is called concrete if there exists a faithful functor $S:\cat C\rightarrow\cat{Sets}.$ We will only consider categories $\cat C$  which contains at least one object $P$ such that $\mor_{\cat C}(P,C)=S(C).$ Examples of such categories are, (1) $\cat{Grps},$ the category of groups which contains the free group on one element $P=\langle x\rangle$ for which the group homomorphisms $P\rightarrow G$ is in one to one correspondence with the elements in the group $G,$ (2) $\cat{Rings},$ the category of rings with unit, which contains the free polynomial ring $P=\mathbb Z[x]$ in one variable over $\mathbb Z$ for which the ring homomorphisms $P\rightarrow R$ is in one to one correspondence with the elements in the ring $R.$

For such categories, a contravariant functor $F:\cat{Top}(X)\rightarrow\cat C$ is called a presheaf of $\cat C$-objects on $X.$  With the assumption on the category $\cat C$ we can formulate Definition II.(1.2),\cite{HH77}, in Hartshorne the following way

\begin{definition} A presheaf $\mathcal F$ of abelian groups on $X$ is called a sheaf if the following two conditions hold:
\begin{itemize}
\item[(1)] If $U$ is an open set, if $\{V_i\}_{i\in I}$ is an open covering of $U,$ and if $s:P\rightarrow\mathcal F(U)$ satisfies $s|_{V_i}=0$ for all $i,$ then $s=0$;
\item[(2)] If $U$ is an open set, if $\{V_i\}_{i\in I}$ is an open covering of $U,$ and if we have $s_i:P\rightarrow\mathcal F(V_i)$ for each $i\in I,$ with the property that for each $i,j,\ s_i|_{V_i\cap V_j}=\ s_j|_{V_i\cap V_j},$ there is an $s:P\rightarrow \mathcal F(U)$ such that $s|_{V_i}=s_i$ for each $i\in I.$
\end{itemize}
\end{definition}

\begin{lemma}\label{asssheaf} Let $F$ be a presheaf on $X.$ Then there exists a sheaf $\mathcal F$ on $X$ with a morphism $\theta:F\rightarrow\mathcal F$ such that if $\mathcal G$ is another sheaf with a morphism $\phi:F\rightarrow\mathcal G,$ there is a unique morphism $\psi:\mathcal F\rightarrow\mathcal G$ such that $\phi=\psi\circ\theta.$
\end{lemma}

\begin{proof} We prove that the presheaf $\mathcal F$ defined by $\mathcal F(U)=\underset{\underset{V\subsetneq U}\longleftarrow}\lim\ F(V)$ is a sheaf.

\begin{itemize} \item[(1)] If $s:P\rightarrow\mathcal F(U)$ restricts to $s|_{V_i}=0$ for each $i\in I,$ then $s$ has to be equal to $0:P\rightarrow\mathcal F(U)$ by the universal property of projective limits.
\item[(2)] The condition says that for each subset $V\subsetneq U$ we have a morphism $s_V:P\rightarrow F(V)\rightarrow\mathcal F(V)$ commuting in all chains of open subsets. Then there exists a unique morphism $s:P\rightarrow\mathcal F(U)$ such that $s|_V=s_V.$ 
\end{itemize}
By the definition of projective limits, it follows that the sheaf $\mathcal F$ has the proposed condition.
\end{proof}

\begin{definition} Let $X$ be a topological space and $F$ a presheaf on $X.$ Then the sheaf $\mathcal F$ on $X$ in Lemma \ref{asssheaf}, is called the sheaf associated to the presheaf $F.$
\end{definition}

\begin{definition} Let $F$ be a presheaf on a topological space $X.$ Then for $x\in X$ we define the stalk of $F$ in $x$ as $F_x=\underset{\underset{x\in U}\longrightarrow}\lim\ F(U).$
\end{definition}

\begin{lemma}\label{stalklemma} Let $F$ be a presheaf on $X$ and let $\mathcal F$ be the associated sheaf. Then  for all $x\in X,$ $\mathcal F_x=F_x.$
\end{lemma}

\begin{proof} By the universal property there exists an isomorphism $F_x\simeq\mathcal F_x.$
\end{proof}

\section{Affine Schemes}

Let $A$ be a commutative ring. Then the set $X=\spec A=\{\mathfrak p\subset A|\mathfrak p\text{ is a prime ideal}\}$ is given the Zariski topology by letting the closed sets be $$V(\mathfrak a)=\{\mathfrak p\in\spec A|\mathfrak a\subseteq\mathfrak p\}$$ for $\mathfrak a\subseteq A$ an ideal. For an element $f\in A$ we define $$D(f)=\{\mathfrak p\in\spec A|f\notin\mathfrak p\},$$ and the collection of sets $\{D(f)\}_{f\in A}$ is a basis for the topology on $X=\spec A.$ 

For a prime ideal $\mathfrak p\subset A$ we have a unique ring homomorphism $\rho_{\mathfrak p}:A\rightarrow A_{\mathfrak p},$ from $A$ to the localization of $A$ in $\mathfrak p,$ such that $\rho_{\mathfrak p}(a)$ is invertible whenever $a\notin\mathfrak p.$ 

Let $\tilde O_X$ be the presheaf on $X=\spec A$ defined by $\tilde O_X(U)=\prod_{\mathfrak p\in U}A_{\mathfrak p}$ 
and consider the ring homomorphism $\rho=\prod\rho_{\mathfrak p}:A\rightarrow \tilde O_X(U).$ We let $O_X(U)$ be the subring of $\tilde O_X(U)$ generated by the subset $\rho(A)$ together with the elements $\rho(a)^{-1}$ whenever $\rho(a)$ is invertible.

\begin{lemma} We have that $O_X$ is a sheaf on $X$ and so $\mathcal O_X=O_X.$
\end{lemma}

\begin{proof} If $\rho(a)$ is invertible, $\rho(a)\notin\mathfrak p$ for all $\mathfrak p\in U.$ Thus we can write any element in $O_X(U)$ as a sequence $(\rho(a)\cdot\rho(f)^{-1})_{\mathfrak p\in U}\in O_X(U)$ where $\rho(f)$ is invertible in each $A_{\mathfrak p},$ that is $\rho(f)\notin\mathfrak p,\mathfrak p\in U.$ From this it follows that $O_X$ is a sheaf.
\end{proof}

Because of uniqueness of projective limits, it follows that the sheaf $\mathcal O_X$ associated to $O_X$ maps to $\tilde{\mathcal O}_X.$

\begin{definition} The sheaf $\mathcal O_X$ of rings on $X=\spec A$ is called the sheaf of regular functions on $X.$
\end{definition}

\begin{proposition}\label{sorprop} Let $A$ be a ring and $X=\spec A.$ Then
\begin{itemize}
\item[(1)] $\mathcal O_{X,\mathfrak p}\simeq A_p,$
\item[(2)] $\mathcal O_X(D(f))\simeq A_f.$
\end{itemize}
\end{proposition}

\begin{proof}
(1) follows directly from Lemma \ref{stalklemma}. So we prove (2) by proving that $O_X(D(f))\simeq A_f.$ Let $\phi:A_f\rightarrow O_X(D(f))$ be defined by $$\phi(\frac{a}{f^n})=(\frac{a}{f^n})_{\mathfrak p\in D(f)}\subseteq\prod_{\mathfrak p\in D(f)}A_{\mathfrak p}.$$
This is a ring homomorphism and we prove that it is injective:
Assume $\frac{a}{f^n}=0$ in all $A_{\mathfrak p}.$ Then there exists an $h\notin\mathfrak p$ such that $ha=0$ for all $f\notin\mathfrak p.$ Thus $h\in\ann(a),$ and for all $f\notin \mathfrak p, h\notin\mathfrak p.$ Thus $V(\ann(a))\cap D(f)=\emptyset$ in $A.$ Thus $f\in\sqrt{\ann a}$ so that $f^m\in\ann a$ for some power $m.$ This says that $\frac{f^ma}{f^n}=0$ so that $0=\frac{a}{f^n}\in A_f$ by definition. 

Now we prove that it is surjective: So let $s=\eta(a)\eta(b)^{-1}\in O(D(f)).$ Then we can assume that $s=(\frac{a}{b})_{f\notin\mathfrak p}\in\prod_{f\notin\mathfrak p}A_{\mathfrak p}$ where $b\notin\mathfrak p$ for any prime $\mathfrak p$ not containing $f.$ This says that $f\notin\mathfrak p\Rightarrow b\notin\mathfrak p.$ Thus $$D(f)\subset D(b)\Rightarrow V(b)\subset V(f)\Rightarrow\sqrt{f}\subset\sqrt{b}\Rightarrow f^n=hb$$ for some $n$ and some $h\in A.$  Notice that $h\notin\mathfrak p,$ because otherwise we would have $f^n=hb\in\mathfrak p.$ This says $\frac{a}{b}=\frac{ha}{hb}=\frac{ha}{f^n}$ in  each $A_{\mathfrak p},$ proving the surjectivity.
\end{proof}

Point (2) in Proposition \ref{sorprop} suggest a more straightforward definition of the sheaf of rings on an affine scheme.

\begin{corollary} For every open subset $U\subseteq\spec A=X$ we have  $$\mathcal O_X(U)=\underset{\underset{D(f)\subseteq U}\longleftarrow}\lim\ A_f.$$
\end{corollary}

\begin{proof} This follows from point (2) in Proposition \ref{sorprop} which states that $\mathcal O_X(D(f))\simeq A_f.$
\end{proof}

\begin{definition} A Scheme $(X,\mathcal O_X)$ is a topological space $X$ together with a sheaf of rings $\mathcal O_X$ which has a covering of open sets $U\subseteq X$ such that $$(U,\mathcal O_X|_U)\simeq (\spec A_U,\mathcal O_{\spec A_U})$$ for rings $A_U.$
\end{definition}

When $f:X\rightarrow Y$ is a continuous function and $\mathcal F_X$ is a sheaf on $X,$ then the presheaf defined by $f^\ast\mathcal F_X(U)=\mathcal F(f^{-1}(U))$ is a sheaf on $Y.$

\begin{lemma} Let $\phi:A\rightarrow B$ be a ring homomorphism. Then the map $$f:\spec B\rightarrow\spec A,\ f(\mathfrak p)=\phi^{-1}(\mathfrak p)\in\spec A$$ is continuous and induces a morphism $\phi^\ast: f^\ast \mathcal O_{\spec B}\rightarrow \mathcal O_{\spec A} $ of sheaves on $\spec A.$ 
\end{lemma}

\begin{proof} For each closed subset $V(\mathfrak a)\subseteq\spec A$ we have $f^{-1}(V(\mathfrak a))=V(\phi^{-1}(\mathfrak a))\subseteq\spec B.$
Also, we have a morphism of presheaves $\tilde\phi^\ast:f^\ast O_{\spec B}  \rightarrow O_{\spec A} $ by sending each prime $\mathfrak \phi^{-1}(p)\in f^{-1}(U)\subseteq\spec B$ to its corresponding component $\mathfrak p\in U\subseteq\spec A.$ By the universal property of projective limits, this defines the proposed morphism $\phi^\ast: f^\ast \mathcal O_{\spec B}\rightarrow \mathcal O_{\spec A} $ of sheaves on $\spec A.$ 
\end{proof}

\begin{definition} A morphism of schemes $f: (X,\mathcal O_X)\rightarrow (Y,\mathcal O_Y)$ is a continuous morphism $f:X\rightarrow Y$ such that every time $f$ maps $U=\spec B\subseteq X$ into $V=\spec A\subseteq Y,$ the restricted morphism $f|_U$ should be equal to a morphism defined by a ring homomorphism $\phi:A\rightarrow B.$
\end{definition}

\end{document}